\documentclass[12pt]{article}
\usepackage{mathrsfs}%{article}
\usepackage{graphicx}
\usepackage{enumerate}
\makeatother
\usepackage{amsfonts,amsmath,amssymb}
\usepackage{epstopdf}
\usepackage{caption}
\usepackage{subcaption}

\title{}\date{}
\title{Optical Vortex Solitons Propagating in Pair Plasmas: Existence and Computation}

\author{Luciano Medina\\Department of Mathematics\\New York University\\Tandon School of Engineering\\Brooklyn, New York 11201, USA}

\def\XXint#1#2#3{{\setbox0=\hbox{$#1{#2#3}{\int}$}
 \vcenter{\hbox{$#2#3$}}\kern-.5\wd0}}

\newtheorem{oldtheorem}{Theorem}[section]
\newtheorem{oldassertion}[oldtheorem]{Assertion}
\newtheorem{oldproposition}[oldtheorem]{Proposition}
\newtheorem{oldremark}[oldtheorem]{Remark}
\newtheorem{oldlemma}[oldtheorem]{Lemma}
\newtheorem{olddefinition}[oldtheorem]{Definition}
\newtheorem{oldclaim}[oldtheorem]{Claim}
\newtheorem{oldcorollary}[oldtheorem]{Corollary}

\newenvironment{theorem}{\begin{oldtheorem}$\!\!\!${\bf.}}{\end{oldtheorem}}

\newenvironment{lemma}{\begin{oldlemma}$\!\!\!${\bf.}}{\end{oldlemma}}

\newbox\qedbox

\newenvironment{proof}{\smallskip\noindent{\bf Proof.}\hskip \labelsep}%
                        {\hfill\penalty10000\copy\qedbox\par\medskip}

\begin{document}
\maketitle

\begin{abstract}
The dynamics of short intense electromagnetic pulses propagating in a relativistic pair plasma with small temperature asymmetry is of current interest. Such pulses were shown to be governed by a nonlinear Schrodinger equation with a new type of focusing-defocusing nonlinearity. We provide an existence theory and numerical analysis for a special class of solutions, known as ``ring-profiled'' optical vortex solitons. Our methods include both a direct and constrained minimization. We give a necessary condition for the existence of nontrivial solutions and a series of interesting estimates in terms of the relevant parameters describing the electromagnetic pulse. Additionally, we use the constrained minimization approach and a finite element formalism to numerically study the behavior of the profiles of soliton solutions.
\end{abstract}

\section{Introduction}
\setcounter{equation}{0}
Optical vortices and vortex solitons permeate through numerous branches of fundamental and applied physics. They appear as topological defects in Bose-Einstein condensates and as phase singularities in wave propagation. The amount of literature focusing on optical vortices and vortex solitons in Bose-Einstein condensates is extensive \cite{A,DKT,KMT}. An entire discipline in physics, called singular optics, focusing on wave singularities emerged from the study of optical vortices. Some applications include areas such as quantum information processing, wireless communications, particle interactions, and cosmology  \cite{ABSW,BKKH,BSV,KMT,RSS,SO,SL,YW}. 

The theoretical description of a complex-valued light wave propagating in a nonlinear medium and governed by a nonlinear Schr\"odinger equation brings about many nonlinear problems \cite{DY,KVT1,KVT2,MSZ,NAO,RLS,SF}.  Such nonlinear problems present interesting mathematical challenges to the mathematical analysts \cite{AC,ASY,CL,LiuRen, FS, YZ,YZ2}. It is the nonlinearity appearing in such equations that help describe the properties of the structures that an electromagnetic active medium can support. The understanding and identification of such nonlinearities plays an important role in scientific investigation and even discovery.

In this work, we establish a rigorous mathematical theory for a new type of nonlinearity arising in the study of the nonlinear propagation of an electromagnetic wave in pair plasmas by Mahajan, Shatashvili, and Berezhiani in \cite{Mahajan} and, then, also considered in \cite{BMS,BMS2}. Such mathematical theory was pursued by Zhang in \cite{Zhang}.

The governing equation, of the nonlinear Schrodinger type, describing the nonlinear evolution of the vector potential of an electromagnetic pulse propagating in an arbitrary pair plasma with temperature asymmetry, in normalized form, is 
\begin{align}
i\dfrac{\partial A}{\partial t}+\dfrac{\partial^2 A}{\partial \xi^2}+\nabla_{\perp}^2 A+f(|A|^2)A=0.\label{0.1}
\end{align}
In the above, $A$ is the slowly varying amplitude of the circularly polarized electromagnetic pulse, $\nabla_{\perp}^2=\partial^2/\partial x^2+\partial^2/\partial y^2$ is the diffraction operator, and $\xi$ is the co-moving coordinate.

The new focusing-defocusing nonlinearity is of the form 
\begin{align}
f(|A|^2)=\dfrac{|A|^2}{\left(1+|A|^2\right)^2},\label{0.2}
\end{align}
where $I=|A|^2$ is the intensity of the  field. This type of nonlinearity belongs to the class of saturable or saturating nonlinearities  (whose magnitude tends to a constant as the field intensity becomes large) \cite{BMS,DKT,Medina, SF}. As shown in \cite{Mahajan}, stable localized and de-localized wave structures are supported in spatial dimensions $D=1,2,3$ by an electromagnetic medium described by such saturable nonlinearity. In \cite{BMS}, numerical results were presented of the formation of 2-dimensional soliton structures carrying a screw type dislocation, also called ``optical vortices''. 

Spatial solutions of \eqref{0.1}, which do not change their intensity profile during propagation, are described under the radially symmetric ansatz
\begin{align}
A=A(r)\exp(i\lambda t+im\theta),\label{0.3}
\end{align} 
where $\lambda\in\mathbb{R}$ is the nonlinear frequency shift and $m\in\mathbb{Z}$ defines the topological charge of vortices. The amplitude of localized solutions satisfies $A(0)=0$ and decays to zero at infinity. We impose the `boundary' condition $A(R)=0$ for a sufficiently large distance $R>0$ away from the ``vortex core''. Such solutions are termed ``ring-profiled'' optical vortex solitons. Under the ansatz \eqref{0.3}, together with the boundary conditions, we get the two-point boundary value problem 
\begin{align}
\left\{\begin{array}{c}
\dfrac{d^2A}{dr^2}+\dfrac{1}{r}\dfrac{dA}{dr}-\dfrac{m^2}{r^2}A-\lambda A+\dfrac{A^3}{(1+A^2)^2}=0,\qquad r\in(0,R),\\
A(0)=0=A(R).\label{vortexEq}
\end{array}\right.
\end{align}
In terms of the radially symmetric ansatz, an important parameter characterization of spatial optical vortex solitons is
\begin{align}
N(A)=2\pi\int_0^RrA^2dr. 
\end{align}
Depending on the context, $N(A)$ may represent the beam power \cite{KVT2,KA}, energy flux \cite{SF}, or stability integral \cite{BM}. In our scenario, we will refer to $N(A)$ as the photon number (or soliton energy) as in \cite{Mahajan}.

In analogy with the nonconservative motion of a particle, \eqref{vortexEq} is rewritten as 
\begin{align}
\dfrac{d}{dr}\left[\left(\dfrac{dA}{dr}\right)^2+V(A)\right]=\dfrac{m^2}{r^2}\dfrac{dA^2}{dr}-\dfrac{2}{r}\left(\dfrac{dA}{dr}\right)^2,\label{0.5}
\end{align}
where the ``effective potential'' is $V_{\lambda}(A)=-\lambda A^2+\ln(1+A^2)-A^2/(1+A^2)$. 
The $m$-vortex equation \eqref{vortexEq} is then the Euler-Lagrange equation of the action functional
\begin{align}
I_{\lambda}(A)=\dfrac{1}{2}\int_0^R\left\{rA_r^2+\dfrac{m^2}{r}A^2-rV_{\lambda}(A)\right\}dr.\label{actionFunc}
\end{align}
We define the `energy' functional 
\begin{align}
\mathcal{E}(A)=\int_0^R\left\{rA_r^2+\dfrac{A^2}{r}\right\}dr\label{energy}
\end{align}
and seek nontrivial `finite energy' solutions  $(\mathcal{E}(A)<\infty)$ of \eqref{vortexEq}. 

The rest of the paper is summarized as follows. In section 2, for a more general nonlinearity, we prove a necessary condition for the existence of nontrivial solutions of \eqref{vortexEq}. In section 3, we use a direct minimization to prove the existence of positive solutions of \eqref{vortexEq}. In this approach, the nonlinear frequency shift may be prescribed in a continuous range of values. In section 4, a constrained minimization problem, where the photon number serves as the constraint, is used to prove the existence of positive solutions. In this scenario, the nonlinear frequency shift is undetermined and the photon number may be prescribed. The constrained minimization is then used to obtain non-existence of ``small-photon-number'' pulses and estimates on the nonlinear frequency shift in terms of the parameters describing the electromagnetic pulse. In section 5, we again take advantage of the constrained minimization and use a finite element formalism to numerically compute the profiles of soliton solutions subject to a prescribed photon number. We study the behavior of the profiles of soliton solutions  by varying the photon number and topological charge. A summary of our results is provided in section 6.

% % % % % % % % % % % % % % % % % % % % % % % % % % % % % % % % % %
% % % % % % % % % % % % % % % % % % % % % % % % % % % % % % % % % %
\section{Necessary condition for nontrivial solutions}
\setcounter{equation}{0}
As in Vakhitov and Kolokolov \cite{VK}, we give a necessary condition for the existence of nontrivial finite energy solutions of \eqref{0.1}, under the radially symmetric ansatz \eqref{0.3}, and for a medium with a bounded nonlinearity $f(I)$.

Let $f:[0,\infty)\rightarrow\mathbb{R}$ be a bounded function, i.e., there is some $M>0$ such that $|f(I)|\leq M$ for all $I\in [0,\infty)$. Under the radially symmetric ansatz \eqref{0.3}, the $m$-vortex equation \eqref{vortexEq} becomes 
\begin{align}
\left\{\begin{array}{c}
\dfrac{d^2A}{dr^2}+\dfrac{1}{r}\dfrac{dA}{dr}-\dfrac{m^2}{r^2}A-\lambda A+f(A^2)A=0,\qquad r\in(0,R),\\
A(0)=0=A(R).\label{vortexEq2}
\end{array}\right.
\end{align}
\begin{theorem} Let $(A,\lambda)$ be a nontrivial finite energy solution pair of the $m$-vortex equation \eqref{vortexEq2} and $f_{\max}=\sup\{|f(I)|:I\in [0,\infty)\}<\infty$. Then the nonlinear frequency shift satisfies
\begin{align}
\lambda \leq f_{\max}-\dfrac{r_0^2+m^2}{R^2},\label{Ncod}
\end{align}
where $r_0(\approx 2.404825)$ is the first positive zero of the Bessel function $J_0$ \cite{YZ}.
\end{theorem}
\begin{proof}
Similarly to \cite{YZ}, note that $\liminf\limits_{r\rightarrow 0}\left\{rA(r)|A_r(r)|\right\}= 0$. Then we can extract a sequence of real numbers $\{r_j\}_{j=1}^{\infty}$ such that $r_j\rightarrow 0$ as $j\rightarrow \infty$ and $\liminf\limits_{r\rightarrow 0}\left\{r_jA(r_j)|A_r(r_j)|\right\}=0$. Multiplying \eqref{vortexEq2} by $rA$ and integrating by parts, we get
\begin{align}
\int_{0}^{R}\left\{-rA^2_r-\frac{m^2}{r}A^2-\lambda rA^2+rf(A^2)A^2\right\}dr=0. \label{neccEq}
\end{align}
From the inequalities 
\begin{align}
f(A^2)\leq f_{\max}\quad\text{and}\quad \int_0^R rA^2dr\leq R^2\int_0^R \dfrac{A^2}{r}dr,
\end{align}
we have
\begin{align}
0\leq\int_{0}^{R}\left\{-rA^2_r-\left(\frac{m^2}{R^2}+\lambda-f_{\max}\right)rA^2\right\}dr.
\end{align}
Take $A$ to be a radially symmetric function defined on $B_R:=\left\{(x,y)\in\mathbb{R}^2:x^2+y^2\leq R^2\right\}$. The finite energy and boundary conditions then implies that $A$ belongs in the standard Sobolev space $W_0^{1,2}(B_R)$. We then use the Poincar\'e inequality 
\begin{align}
\int_0^R ru^2dr\leq \dfrac{R^2}{r_0^2}\int_0^R ru_r^2dr,
\end{align}
with $r_0$ as defined in \eqref{Ncod}, to obtain
\begin{align}
0\geq\left(\frac{r_0^2+m^2}{R^2}+\lambda-f_{\max}\right)\int_{0}^{R}rA^2dr.\label{NcodIneq}
\end{align}
The necessary condition \eqref{Ncod} then follows from \eqref{NcodIneq}.$\qquad\square$
\end{proof}

Localized optical vortex solitons are exponentially decaying solutions of the $m$-vortex equation \eqref{vortexEq}. The exponential decay estimate follows identically as in \cite{Medina}. For completeness, we state the following lemma,
\begin{lemma}
If $\lambda>0$, then the solution pair $(A,\lambda)$ of \eqref{vortexEq2} satisfies the exponential decay estimate
\begin{equation}
A^2\leq C_{\lambda}\exp(-\sqrt{\lambda}r),
\end{equation}
where $r$ is sufficiently large and $C_{\lambda}$ is a positive constant depending on $\lambda$ only.
\end{lemma}
\textbf{Remark.} From Theorem 2.1 and Lemma 2.2, a necessary condition for the existence of nontrivial spatially localized solitons of \eqref{0.1}, and with nonlinearity \eqref{0.2}, is 
\begin{align}
0<\lambda\leq\dfrac{1}{4}-\dfrac{r_0^2+m^2}{R^2}.
\end{align}
When $R$ goes to infinity, the necessary condition reduces to $0<\lambda<0.25$, which is in agreement with the results of Berezhiani, Mahajan, and Shatashvili in \cite{BMS}.  

% % % % % % % % % % % % % % % % % % % % % % % % % % % % %
\section{Solution via direct minimization}
\setcounter{equation}{0}
Using a direct minimization approach, we prove the existence of nontrivial positive solutions of the $m$-vortex equation \eqref{vortexEq}. In this scenario, the nonlinear frequency shift may be prescribed in a continuous range of values.

Let's begin by letting $H$ be the completion of $X=\left\{A\in\mathcal{C}^1[0,R]| A(0)=0=A(R)\right\}$, the space of differentiable functions over $[0,R]$ which vanish at the two endpoints of the interval. Equip $H$ with the inner product 
\begin{align}
(A,\tilde{A})&=\int_0^R\left\{rA_r\tilde{A}_r+\dfrac{1}{r}A\tilde{A} \right\}dr,\quad A,\tilde{A}\in H,
\end{align}
and define the norm $||A||_H=\sqrt{(A,A)}$. Then $H$ may be viewed as an embedded subspace of $W_0^{1,2}(B_R)$, composed of radially symmetric functions enjoying the property $A(0)=0$ for any $A\in H$. For convenience rewrite the functional \eqref{actionFunc}, defined over $H$, in the form
\begin{align}
I_{\lambda}(A)=\dfrac{1}{2}\int_0^R\left\{rA_r^2+\dfrac{m^2}{r}A^2+\lambda rA^2-r\ln(1+A^2)+r-\dfrac{r}{1+A^2}\right\}dr,
\end{align}
where $\frac{rA^2}{1+A^2}=r-\frac{r}{1+A^2}.$

Let $h_{\lambda}(A)=\lambda A^2-\ln(1+A^2)$. Then $h_{\lambda}(A)\geq 1-\lambda+\ln(\lambda)$ for $\lambda\in(0,1)$. Hence, for any $|m|\geq 1$ the functional $I_{\lambda}$ satisfies the coercive bound
\begin{align}
I_{\lambda}(A)\geq\dfrac{1}{2}||A||_H^2+(1-\lambda+\ln(\lambda))R.\label{coercivebnd}
\end{align}
It is then appropriate to consider the direct minimization problem
\begin{align}
I_0=\min\{I_{\lambda}(A):A\in H\}.\label{minProb}
\end{align}
We first prove that $I_0\neq 0$. Consider the following function,
\begin{align}
A_0(r)=\left\{
\begin{array}{cc}
\frac{b}{a}r, &0\leq r\leq a,\\
\frac{b}{a}(2a-r), &a\leq r\leq 2a,
\end{array}\right.\label{2.28}
\end{align}
with $R=2a$. By direct computation we obtain,
\begin{subequations}
\begin{align}
\int_0^{2a} rA_0^2dr&=\dfrac{2}{3}a^2b^2,\label{intAa}\\
\int_0^{2a} rA_{0,r}^2dr&=2b^2,\\
\int_0^{2a} \dfrac{1}{r}A_0^2dr&=2b^2(2\ln 2-1),\\
\int_0^{2a}r\ln(1+ A_0^2)dr&=2a^2\left(\ln(1+b^2)-2+\dfrac{2}{b}\tan^{-1}(b)\right),\\
\int_0^{2a} \dfrac{r}{1+A_0^2}dr&=\dfrac{2a^2}{b}\tan^{-1}(b),\label{intAe}
\end{align}
\end{subequations}
where $A_{0,r}:=(A_0)_r$. Similarly to Lemma 3.3 in \cite{YZ}, note that $A_0$ may be obtained as the limit of a Cauchy sequence in $H$, consequently, $A_0$ belongs in $H$. Then, using \eqref{intAa}-\eqref{intAe}, we get
\begin{align}
I_{\lambda}(A_0)=a^2\left(3-\dfrac{3}{b}\tan^{-1}(b)-\ln(1+b^2)+\dfrac{1}{3}\lambda b^2+\dfrac{b^2}{a^2}(1+m^2(2\ln(2)-1))\right).
\end{align}
For any $\delta>0$, we may choose $a^2$ sufficiently large such that 
\begin{align}
\dfrac{1}{a^2}(1+m^2(2\ln(2)-1))\leq\dfrac{1}{3}\lambda\delta.
\end{align}
Then, 
\begin{align}
I_{\lambda}(A_0)\leq a^2\left(3-\dfrac{3}{b}\tan^{-1}(b)-\ln(1+b^2)+\dfrac{1}{3}\lambda b^2(1+\delta)\right)=a^2 g_{\lambda}(b),
\end{align}
with $g_{\lambda}(b)=3-\dfrac{3}{b}\tan^{-1}(b)-\ln(1+b^2)+\frac{1}{3}\lambda b^2(1+\delta)$ and $g_{\lambda}(b)<0$ for all $\lambda$ satisfying
\begin{align}
\lambda<\dfrac{3}{(1+\delta)b^2}\left(\ln(1+b^2)+\dfrac{3}{b}\tan^{-1}(b)-3\right)=:\dfrac{3}{1+\delta}k(b).
\end{align}
The function $k(b)$ attains a maximum value $k_{max}\approx 0.0675407$ at $b_0\approx\pm 1.99379$. Hence, when
\begin{align}
0<\lambda <\frac{3}{1+\delta}k_{max}\label{3.11}
\end{align}
and $R$ satisfies
\begin{align}
\dfrac{12(1+m^2(2\ln(2)-1))}{\lambda\delta}\leq R^2,\label{eps}
\end{align}
we have an $A_0\in H$ such that $I_{\lambda}(A_0)<0$. Therefore, $I_0<0$ and $A\equiv 0$ is not a solution of \eqref{minProb}.

Let $\{A_n\}_{n=1}^{\infty}$ be a minimizing sequence of \eqref{minProb}. From the coercive bound \eqref{coercivebnd}, it follows that the sequence is bounded in $H$, i.e., there exists a constant $C>0$ independent of $n$ such that 
\begin{align}
||A_n||_H^2=\int_0^R\left\{rA_{n,r}^2+\dfrac{A^2_n}{r}\right\}dr\leq C.\label{boundedness}
\end{align}
From the reflexivity of the space, and without loss of generality, we may suppose the sequence $\{A_n\}_{n=1}^{\infty}$ converges weakly to an element $A\in H$. Moreover, note that $I_{\lambda}$ is an even functional and, since $A$ is real-valued, $||A|_r|=|A_r|$. Thus, $I_{\lambda}(A_n)=I_{\lambda}(|A_n|)$ and $\{A_n\}_{n=1}^{\infty}$ may be modified to a sequence of nonneqative valued functions.

The compact embedding $W^{1,2}(B_R)\subset\subset L^p(B_R)$ for $p\geq 1$ gives the strong convergence $A_n\rightarrow A$ in $L^p(B_R)$. Hence, $A$ is radially symmetric and satisfies the boundary condition $A(R)=0$.

From \eqref{boundedness} and the inequalities
\begin{align}
\int_0^R r\ln(1+A^2)dr\leq \int_0^R rA^2\leq R^2\int_0^R \dfrac{A^2}{r}dr,
\end{align}
we use Fatou's lemma to get
\begin{subequations}
\begin{align}
\int_0^R rA^2_rdr\leq \liminf_{n\rightarrow\infty}\int_0^R rA^2_{n,r}dr,\label{Fatou1}\\
\int_0^R \frac{A^2}{r}dr\leq \liminf_{n\rightarrow\infty}\int_0^R \frac{A^2_n}{r}dr,\label{Fatou2}\\
\int_0^R r\ln(1+ A^2)dr\leq \liminf_{n\rightarrow\infty}\int_0^Rr\ln(1+A_n^2)dr,\label{Fatou3}\\
\int_0^R \dfrac{r}{1+A^2}dr\leq \liminf_{n\rightarrow\infty}\int_0^R\dfrac{r}{1+A_n^2}dr.\label{Fatou4}
\end{align}
\end{subequations}
We also need to show that $A(0)=0$. Let $\{A_n\}$ be a sequence in $W^{1,2}(\epsilon,R)$ where $\epsilon\in(0,R)$. For any $\epsilon\in(0,R)$, $\{A_n\}$ is bounded in $W^{1,2}(\epsilon,R)$. The compact embedding $W^{1,2}(\epsilon,R)\subset\subset C[\epsilon,R]$ then gives the convergence $A_n\rightarrow A$ uniformly over $[\epsilon,R]$. Thus, for any pair $r_1,r_2\in (0,R)$ such that $r_1<r_2$ and using $C$ from \eqref{boundedness}, we get
\begin{align}
|A_n^2(r_2)-A_n^2(r_1)|&=\left|\int_{r_1}^{r_2}(A^2_n(r))_rdr\right|\\
&\leq\int_{r_1}^{r_2}2|A_n(r)A_{n,r}(r)|dr\nonumber\\
&\leq 2 \left(\int_{r_1}^{r_2}rA_{n,r}^2(r)dr\right)^{1/2}\left(\int_{r_1}^{r_2}\dfrac{A_n^2(r)}{r}dr\right)^{1/2}\nonumber\\
&\leq 2 C^{1/2}\left(\int_{r_1}^{r_2}\dfrac{A_n^2(r)}{r}dr\right)^{1/2}.\nonumber
\end{align}
Taking the limit as $n\rightarrow\infty$, we have
\begin{align}
|A^2(r_2)-A^2(r_1)|\leq 2 C^{1/2}\left(\int_{r_1}^{r_2}\dfrac{A^2(r)}{r}dr\right)^{1/2}.\label{A0}
\end{align}
From \eqref{Fatou2} and \eqref{boundedness} we get that $A^2/r$ is in $L(0,R)$. Thus the right hand side of \eqref{A0} goes to zero as $r_1,r_2\rightarrow 0$, which gives the existence of the following limit
\begin{equation}
\eta_0=\lim_{r\rightarrow 0}A^2(r)=0.
\end{equation}
As a consequence, the boundary condition $A(0)=0$ is achieved. 

From \eqref{Fatou1}-\eqref{Fatou4}, we conclude that the functional $I_{\lambda}$ is weak lower semicontinuous and
\begin{align}
I_{\lambda}(A)\leq\liminf_{n\rightarrow\infty}I_{\lambda}(A_n)=I_0.\label{lowersemi}
\end{align}
Consequently, the function $A$ obtained as the limit of the minimizing sequence $\{A_n\}_{n=1}^{\infty}$ is a solution to the minimization problem \eqref{minProb}. Standard methods may then be used to show that $A$ is a classical solution of \eqref{vortexEq}.

Moreover, suppose that there is a point $r_0\in(0,R)$ such that $A(r_0)=0$. Then $r_0$ would be a minimum point for $A(r)$ and $A_r(r_0)=0$. However, by the uniqueness theorem of the initial value problem of ordinary differential equations, $A(r)\equiv 0$ for all $r\in(0,R)$, contradicting $I_0<0$. Hence, $A(r)>0$ for all $r\in(0,R)$. 

For simplicity, we may take $\delta=1$ in \eqref{eps} and summarize our results with the following theorem, 
\begin{theorem}
For each nonlinear frequency shift $\lambda\in(0,\frac{3}{2}k_{max})$, $|m|\geq 1$, and
\begin{align}
\dfrac{6(1+m^2(2\ln(2)-1))}{\lambda}\leq R^2,
\end{align}
there exists a solution pair $(A,\lambda)$, such that $A(r)>0$ for $r\in(0,R)$, to the $m$-vortex equation \eqref{vortexEq}.
\end{theorem}
\textbf{Remarks.} Numerically, and using an analogy with the nonconservative motion of a particle, Berezhiani, Mahajan, and Shatashvili in \cite{BMS}, showed that localized optical vortex solitons may exist in the range $0<\lambda<\lambda_{cr}\approx 0.2162$. As a consequence of Theorem 3.1, if we let $\delta\rightarrow 0$ in \eqref{eps}, then $R\rightarrow\infty$ and existence can be expected in the range $0<\lambda<3k_{max}\approx 0.2026221$, which is in agreement with the results in \cite{BMS} and covers over 93.7\% of the predicted possible range of values for $\lambda$.

% % % % % % % % % % % % % % % % % % % % % % % % % % % % %
\section{Constrained minimization approach}
\setcounter{equation}{0}
We use a constrained minimization approach to establish additional results in relation to the solutions and governing parameters of the $m$-vortex equation \eqref{vortexEq}. Particularly, we obtain a lower bound for the nonlinear frequency shift in terms of a prescribed photon number $N_0$ and other relevant parameters. We also prove non-existence of nontrivial ``small-photon-number'' pulses and provide a lower bound on the magnitude of the topological charge which guarantees a negative nonlinear frequency shift. 

To this end, treat \eqref{vortexEq} as a nonlinear eigenvalue problem, which may be viewed as the Euler-Lagrange equations corresponding to the functional $J$ given by
\begin{equation}
J(A)=\frac{1}{2}\int_{0}^{R}\left\{rA^2_r+\frac{m^2}{r}A^2-r\ln(1+A^2)+\dfrac{rA^2}{1+A^2}\right\}dr,\label{Jfun}
\end{equation}
and the photon number constraint functional,
\begin{equation}
N(A)=2\pi\int_{0}^{R}rA^2dr.\label{constraint}
\end{equation}
Note that the nonlinear frequency shift, $\lambda$, appears as a Lagrange multiplier of the constrained minimization problem. 

Consider the nonempty admissible class
\begin{equation}
\mathcal{A}=\left\{ \text{$A(r)$ is absolutely continuous over $[0,R]$}, A(0)=A(R)=0,\mathcal{E}(A)<\infty\right\},\label{admisClass}
\end{equation}
where $\mathcal{E}(A)$ is a defined by \eqref{energy}. It suffices to show that a solution to the following exists:
\begin{equation}
J_0=\inf_{A\in\mathcal{A}}\left\{J(A)\bigg| N(A)=N_0>0\right\}\label{ConstrndMinProb},
\end{equation}
where $N_0$ is a prescribed value for the photon number.

For an undetermined nonlinear frequency shift, we have the following existence theorem,
\begin{theorem}
Consider the $m$-vortex equation \eqref{vortexEq} subject to the prescribed photon number $N(A)=N_0>0$ and parameters $|m|\geq 1$, $R>0$. There exists a solution pair $(A,\lambda)$ such that $A(r)>0$ for $r\in(0,R)$.
\end{theorem}
\textbf{Proof.}
The functional $J(A)$ satisfies the following coercive bound,
\begin{equation}
J(A)\geq\frac{1}{2}\int\limits_{0}^{R}rA^2_rdr+\int\limits_{0}^{R}\dfrac{A^2}{r}dr-\dfrac{N_0}{2\pi},\label{coercBound}
\end{equation} 
where we have used the basic inequality $\ln(1+x^2)\leq x^2$ for all $x\in\mathbb{R}$, the constraint $N(A)=N_0$, and $|m|\geq 1$. Thus, the minimization problem \eqref{ConstrndMinProb} is well-defined. 

Let $\{A_n\}_{n=1}^{\infty}$ be a minimizing sequence of \eqref{ConstrndMinProb}. From the coercive bound and since $\{A_n\}_{n=1}^{\infty}$ is minimizing, it follows that there exist $C>0$ independent of $n$ such that 
\begin{equation}
C\geq\int_{0}^{R}\left\{rA_{n,r}^2+\dfrac{A_n^2}{r}\right\}dr.\label{2.16}
\end{equation}
Since $J(A)$ and $N(A)$ are even functionals, and as in Theorem 3.1, without loss of generality, we suppose that the sequence $\{A_n\}_{n=1}^{\infty}$ consists of nonnegative radially symmetric functions defined over the disk $B_R$ and vanish on its boundary. Under the radially symmetric reduced norm,
\begin{align}
||A||^2:=\int_{0}^{R}\left\{rA_r^2+rA^2\right\}dr
\end{align}
and using the inequality
\begin{align}
||A_n||^2=\int_{0}^{R}\left\{rA_{n,r}^2+rA_n^2\right\}\leq \int_{0}^{R}\left\{rA_{n,r}^2+R^2\dfrac{A_n^2}{r}\right\}\leq (1+R^2)C,
\end{align}
we see that the sequence $\{A_n\}_{n=1}^{\infty}$ is bounded in $W_0^{1,2}(B_R)$.

From the reflexivity of the space, without loss of generality, we then assume the weak convergence of $\{A_n\}_{n=1}^{\infty}$ to an element $A\in W^{1,2}_0(B_R)$. It then follows, almost identically to the prove of Theorem 3.1, that $A$ is a positive solution of \eqref{ConstrndMinProb}. The contradiction for the non-triviality of the solution is a consequence of the constraint $N(A)=N_0>0$. Lastly, $\lambda$ appears as the Lagrange multiplier of the constrained minimization problem. 
$\qquad\square$

\begin{theorem}
Consider the $m$-vortex equation \eqref{vortexEq} subject to the prescribed photon number $N(A)=N_0>0$.
\begin{enumerate}[(a)]
\item If $m^2+r^2\lambda >0$ for $r\in[0,R]$, then there exists no nontrivial ``small-photon-number'' solution satisfying $N(A)=N_0\leq 1/2$.
\item For any nontrivial solution pair $(A,\lambda)$ of \eqref{vortexEq}, if the topological charge satisfies $|m|\geq \frac{N_0}{2\pi}$, then $\lambda<0$.
\item Let $(A,\lambda)$ be a nontrivial solution of \eqref{vortexEq}. Then nonlinear frequency shift satisfies
\begin{align}
\lambda\geq -\dfrac{\pi R^2}{2N_0}\left(3-\dfrac{3}{b}\tan^{-1}(b)-\ln(1+b^2)+\dfrac{4b^2}{R^2}(1+m^2(2\ln(2)-1))\right)-1,
\end{align}
with $b^2=\dfrac{3}{\pi R^2}N_0$.
\end{enumerate}
\end{theorem}
\begin{proof}
$(a)$ Let $(A,\lambda)$ be a solution pair of \eqref{vortexEq}. Rearranging \eqref{neccEq}, we get
\begin{align}
-\int_0^R rA_r^2dr-\int_0^R\left(\dfrac{m^2}{r^2}+\lambda\right)rA^2dr\geq-\int_0^R\dfrac{rA^4}{(1+A^2)^2}dr\geq-\int_0^RrA^4dr. \label{rearrange1}
\end{align}
Treat $A$ as a radially symmetric function defined over $\mathbb{R}^2$ with support in $B_R$. From the classical Gagliardo-Nirenberg inequality over $\mathbb{R}^2$, we have
\begin{align}
\int_0^R rA^4dr\leq 4\pi\int_0^R rA^2dr\int_0^R rA^2_rdr.\label{gagliardo}
\end{align}
Using the constraint $N(A)=N_0$ and \eqref{gagliardo} in \eqref{rearrange1}, we arrive at
\begin{align}
(2N_0-1)\int_0^RrA_r^2dr-\int_0^R\left(\dfrac{m^2}{r^2}+\lambda\right)rA^2dr\geq 0. 
\end{align}
Therefore, if $N_0\leq 1/2$ and $m^2/r^2+\lambda>0$ for $r\in(0,R]$, then $A\equiv 0$. 

$(b)$ Suppose $(A,\lambda)$ is a nontrivial solution pair of \eqref{vortexEq}. Use Schwartz's inequality and $A(0)=0$ to get
\begin{align}
A^2(\rho)=\int_0^{\rho} 2A(r)A_{r}(r)dr\leq 2\left(\int_0^RrA_r^2(r)dr\right)^{1/2}\left(\int_0^R\dfrac{A^2(r)}{r}dr\right)^{1/2}.
\end{align}
Multiplying by $\rho A^2$, integrating from $0$ to $R$, and using the constraint $N(A)=A_0>0$, we obtain
\begin{align}
\int_0^R rA^4dr\leq \dfrac{N_0}{\pi}\left(\int_0^Rr A_r^2dr\right)^{1/2}\left(\int_0^R\dfrac{A^2}{r}dr\right)^{1/2}.
\end{align}
From the basic inequality, $ab\leq \epsilon a^2+\frac{b^2}{4\epsilon}$ for every $a,b\in\mathbb{R}$ and $\epsilon>0$, we get
\begin{align}
\int_0^R rA^4dr\leq \epsilon\int_0^Rr A_r^2dr+\dfrac{N_0^2}{4\epsilon \pi^2}\int_0^R\dfrac{A^2}{r}dr.\label{eq411}
\end{align}
Consequently, inserting \eqref{eq411} in \eqref{neccEq}, it follows 
\begin{align}
\lambda \int_{0}^{R}rA^2dr&\leq-\left(1-\epsilon\right)\int_0^R rA_r^2 dr-\left(m^2-\dfrac{N_0^2}{4\epsilon\pi^2}\right)\int_{0}^{R}\dfrac{A^2}{r}dr.
\end{align}
Let $\epsilon=1$ above to conclude that $\lambda< 0$ whenever $|m|\geq N_0/(2\pi)$.

$(c)$ Let $A_0$ be an element in $\mathcal{A}$ such that $N(A_0)=N_0$ and $(A,\lambda)$ a solution pair of \eqref{ConstrndMinProb}. Then, $J(A)\leq J(A_0)$ and 
\begin{align}
\int_0^R\left\{rA_r^2+\dfrac{m^2}{r}A^2\right\}dr\leq 2I(A_0)+\int_0^R\left\{r\ln(1+A^2)-\dfrac{rA^2}{1+A^2}\right\}dr.
\end{align}
Insert the above in \eqref{neccEq} and use the constraint $N(A_0)=N_0$ to get the bound
\begin{align}
\lambda\geq -\dfrac{2\pi}{N_0}J(A_0)-1.\label{lowerbnd}
\end{align}
The constraint and \eqref{intAa} gives $b^2=\dfrac{3}{\pi R^2}N_0$. Then, from \eqref{intAa}-\eqref{intAe}, we get
\begin{align}
J(A_0)=\dfrac{R^2}{4}\left(3-\dfrac{3}{b}\tan^{-1}(b)-\ln(1+b^2)+\dfrac{4b^2}{R^2}(1+m^2(2\ln(2)-1))\right).\label{Ja0}
\end{align}
Inserting \eqref{Ja0} into \eqref{lowerbnd} gives the desire lower bound. $\qquad\square$
\end{proof}

%% % % % % % % % % % % % % % % % % % % % % % % % % % % % %
\section{Computation using finite element formalism}
\setcounter{equation}{0}
We use the constrained minimization problem of Section 4 to compute the profiles of soliton solutions, subject to a prescribed photon number, of the $m$-vortex equation \eqref{vortexEq} . Our approach uses a finite element formalism to reduce the infinite dimensional constrained optimization problem to a finite dimensional one. As the photon number increases, we show that ``flat-top'' localized optical vortex solitons may be created, the nonlinear frequency shift approaches a critical value $\lambda_{cr}$, and the soliton amplitude approaches a critical maximum value $A_{cr}$. Additionally, for a fixed photon number and as the topological charge increases, we show that  both the soliton amplitude and nonlinear frequency shift decrease. 

Recall the admissible class $\mathcal{A}$ defined in \eqref{admisClass}. Let $W$ be a subset of $\mathcal{A}$, composed of $N$ orthonormal functions $\left\{\psi_j\right\}_{j=1}^N$, and defined under the inner product
\begin{align}
\langle A,\tilde{A}\rangle=2\pi\int_0^RrA\tilde{A}dr,\qquad A,\tilde{A}\in\mathcal{A},\label{5.1}
\end{align}
whose form is suggested by the photon number constraint functional \eqref{constraint}. Using the formalism
\begin{align}
A=\sum_{j=1}^{N}a_j\psi_j,\label{5.2}
\end{align} 
we approximate functions $A\in \mathcal{A}$, with $a_1,\ldots,a_N\in\mathbb{R}$. Inserting \eqref{5.2} into the constrained minimization problem \eqref{ConstrndMinProb}, we get the finite dimensional version
\begin{align}
J_0=\min\left\{F(a)=J\left(\sum_{j=1}^{N}a_j\psi_j\right)\bigg|\sum_{j=1}^{N}a_j^2=N_0,\quad a\in\mathbb{R}^N\right\},\label{5.3}
\end{align}
where $a=(a_1,\ldots,a_N)$ and called the variational vector. Since $F$ is a continuous function, defined over a compact set in $\mathbb{R}^N$, the constrained minimization problem \eqref{5.3} is well-defined and has a solution. We make use of MATLAB's Optimization Toolbox \cite{Matlab} and Chebfun \cite{chebfun} to numerically solve \eqref{5.3}. 

To compute the nonlinear frequency shift, consider the existence of a Lagrange multiplier $\xi\in\mathbb{R}$ such that $\langle J'(A),\tilde{A}\rangle=\xi\langle N'(A),\tilde{A}\rangle$. Thus, 
\begin{equation}
\int_0^R\left\{rA_r\tilde{A}_r+\dfrac{m^2}{r}A \tilde{A}-\dfrac{rA^3\tilde{A}}{(1+A^2)^2}\right\}dr=4\pi\xi\int_0^RrA\tilde{A}dr\label{5.6}
\end{equation}
for every $\tilde{A}\in \mathcal{A}$. The weak formulation of the $m$-vortex equation \eqref{vortexEq} gives $\lambda = 4\pi\xi$. The value of $\xi$ may also be obtained directly from the implementation of the constrained optimization problem via MATLAB's fmincon. Additionally, taking $\tilde{A}=A$ in \eqref{5.6} and the prescribed photon number $N(A)=N_0>0$, we have
\begin{align}
\lambda =-\dfrac{2\pi}{N_0}\int_0^R\left\{ rA_r^2+\frac{m^2}{r}A^2 -\frac{rA^4}{(1+A^2)^2}\right\}dr.\label{5.7}
\end{align}
%%The numerical error is estimated by substituting the formalism \eqref{5.2} into
%%\begin{equation}
%%error=\int_0^R\left((rA_r)_r-\dfrac{m^2}{r}A+\dfrac{rA^3}{(1+A^2)^2}-\lambda rA\right)^2 dr.
%%\end{equation}

In the numerical implementation, we use twenty trigonometric basis functions ($N=20$) and, to compare results with those in \cite{BMS, Mahajan}, take $R=40$ and the topological charge $m=1$. Figure 1 illustrates the profile of soliton solutions for varying photon numbers and a fixed topological charge $m=1$. We observe that localized wave structures with flat-top shapes may be obtained for large photon numbers. This was also observe in \cite{BMS, Mahajan} by different methods. In principle, it follows that such flat-top solitons may be created with a width as large as desired by increasing the distance $R$ and the photon number $N_0$. 
\begin{figure*}[h]
\centering
\includegraphics[scale=0.35]{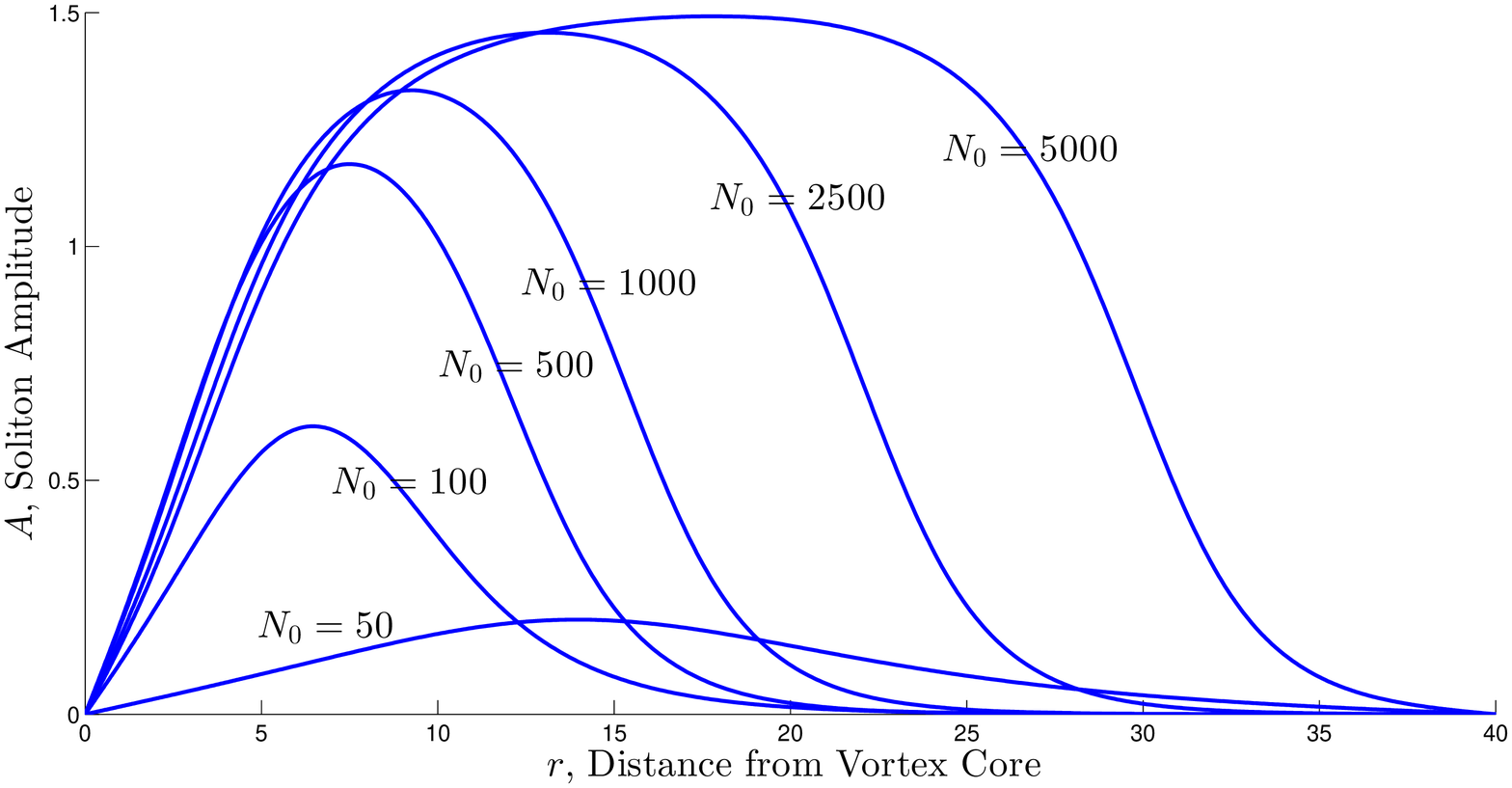}
\caption{Profile of single charged soliton solutions for increasing photon number.}
\label{fig:fig1}
\end{figure*}

In Figure 2, for photon numbers of $N_0=100$ and $N_0=5000$, we present the soliton amplitude over a spatial domain. We observe the formation of the flat-top soliton when the photon number is large. 

\begin{figure*}[h]
\centering
\begin{subfigure}{.5\textwidth}
\centering
\includegraphics[scale=0.27]{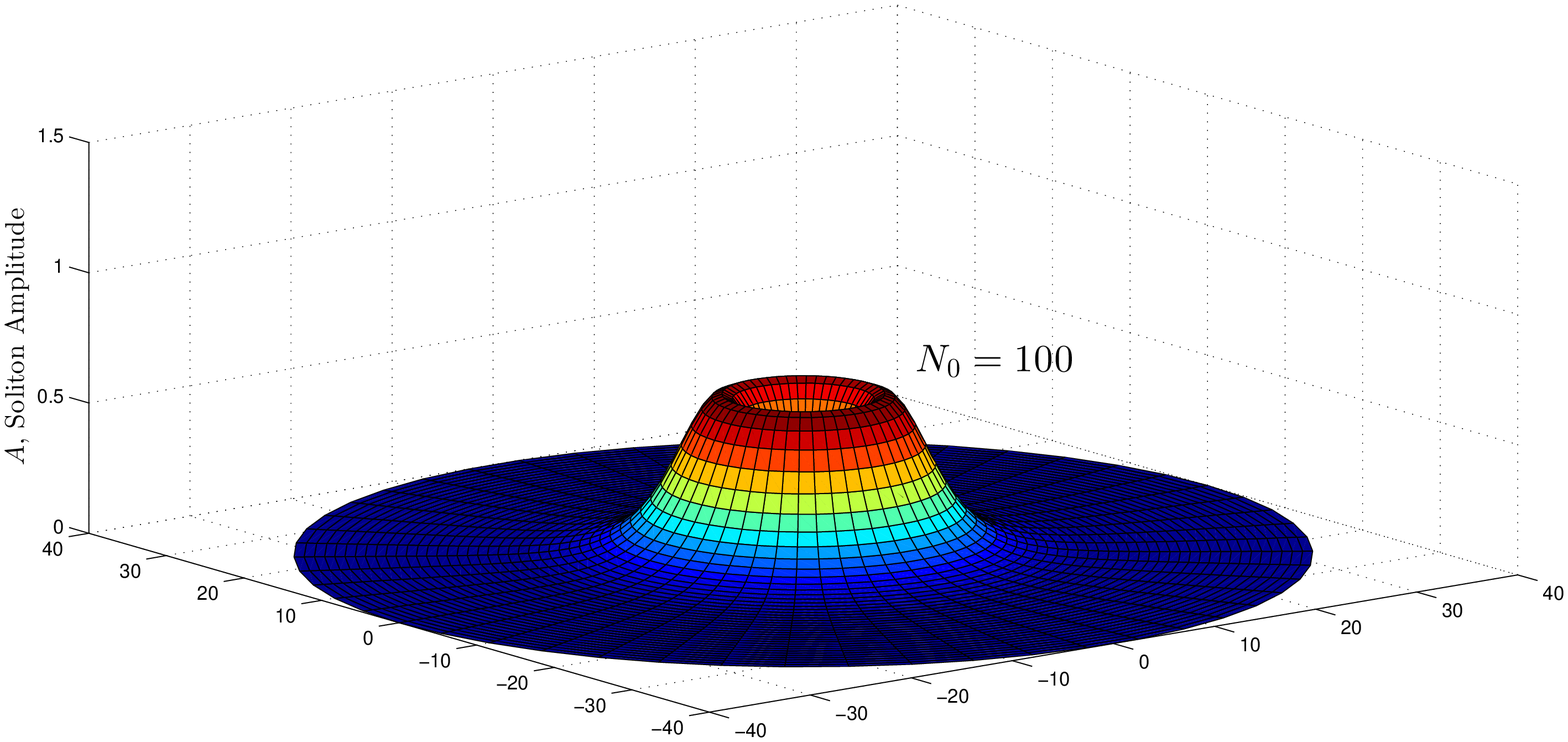}
\label{fig:sub1}
\end{subfigure}%
\begin{subfigure}{.5\textwidth}
\centering
\includegraphics[scale=0.27]{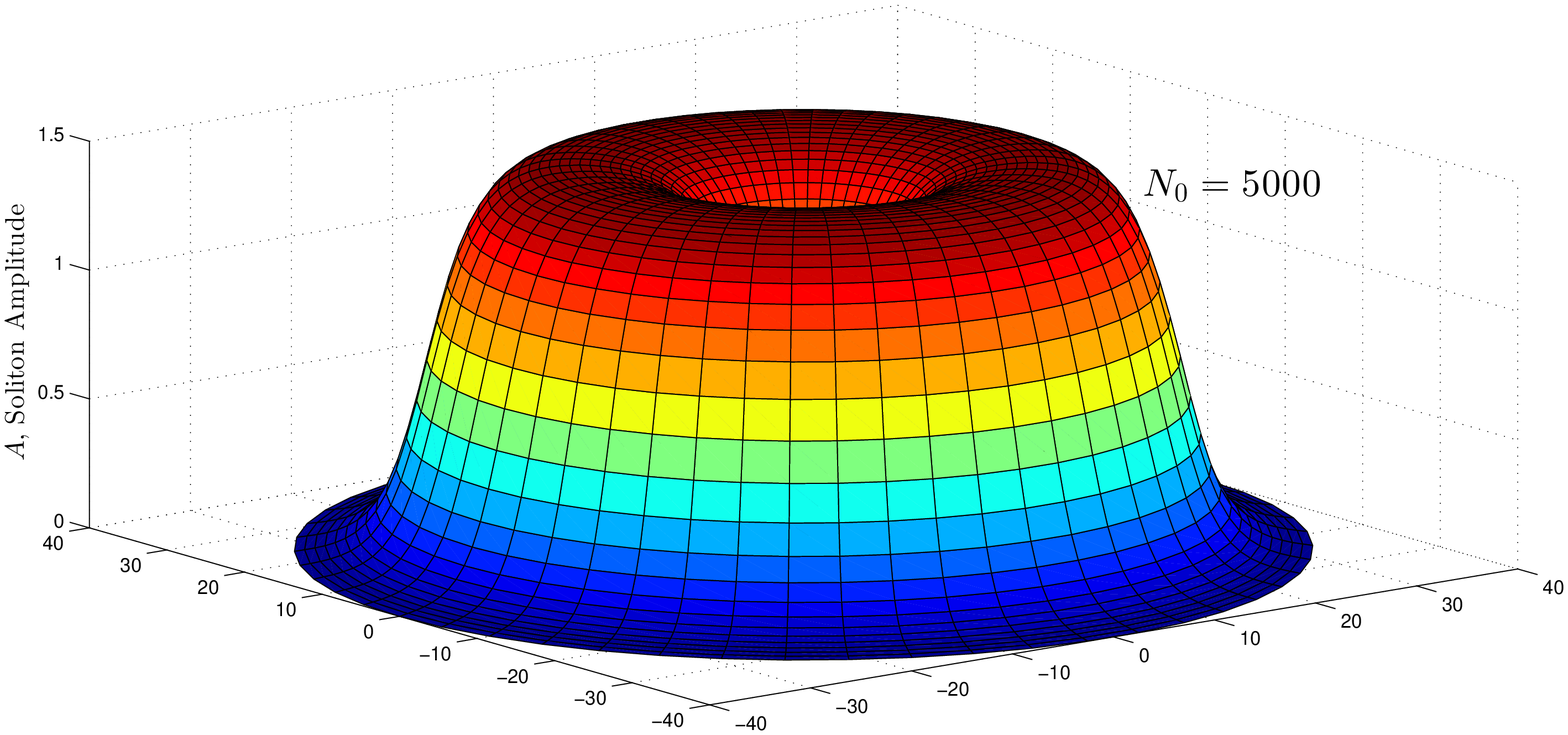}
\label{fig:sub2}
\end{subfigure}
\caption{Soliton amplitude for $N_0=100$ and $N_0=5000$.}
\label{fig:test}
\end{figure*}

Table 1 then shows the corresponding nonlinear frequency shift and maximum soliton amplitude for each vortex number. As the photon number is increased, the nonlinear frequency shift increases and approaches the critical value $\lambda_{cr}\approx 0.2162$ (remarked at the end of Section 3). Moreover, the soliton amplitude increases and seems to be bounded above by a certain critical value $A_{cr}\approx 1.5$. Our numerical results are in agreement with those presented in \cite{BMS, Mahajan}.

\begin{table*}[h]
\centering
\label{table1}
\begin{tabular}{|l|l|l|l|l|l|l|}
\hline
$N_0$ & 50 & 100 & 500 & 1000 & 2500 & 5000 \\ \hline
$\lambda$ & 0.0123 & 0.0901 & 0.1816 & 0.1960 & 0.2053 & 0.2089 \\ \hline
$A_{max}$ & 0.2024 & 0.6157 & 1.1761 & 1.3338 & 1.4571 & 1.4920 \\ \hline
%error & 0.0018 & 0.0063 & 0.0355 & 0.0573 & 0.0473 & 0.0563 \\ \hline
\end{tabular}
\caption{Nonlinear frequency shift and amplitude of singly charged soliton solutions for increasing photon number.}
\end{table*}

We also consider the behavior of the soliton profile for a fixed photon number and a varying vortex topological charge. We fix $N_0=500$ and in Figure 2 display the soliton amplitude as we increase the vortex topological charge. Table 2 shows that both the nonlinear frequency shift and maximum amplitude of the soliton decrease as the topological charge increases. 

\begin{figure*}[h]
\centering
\includegraphics[scale=0.35]{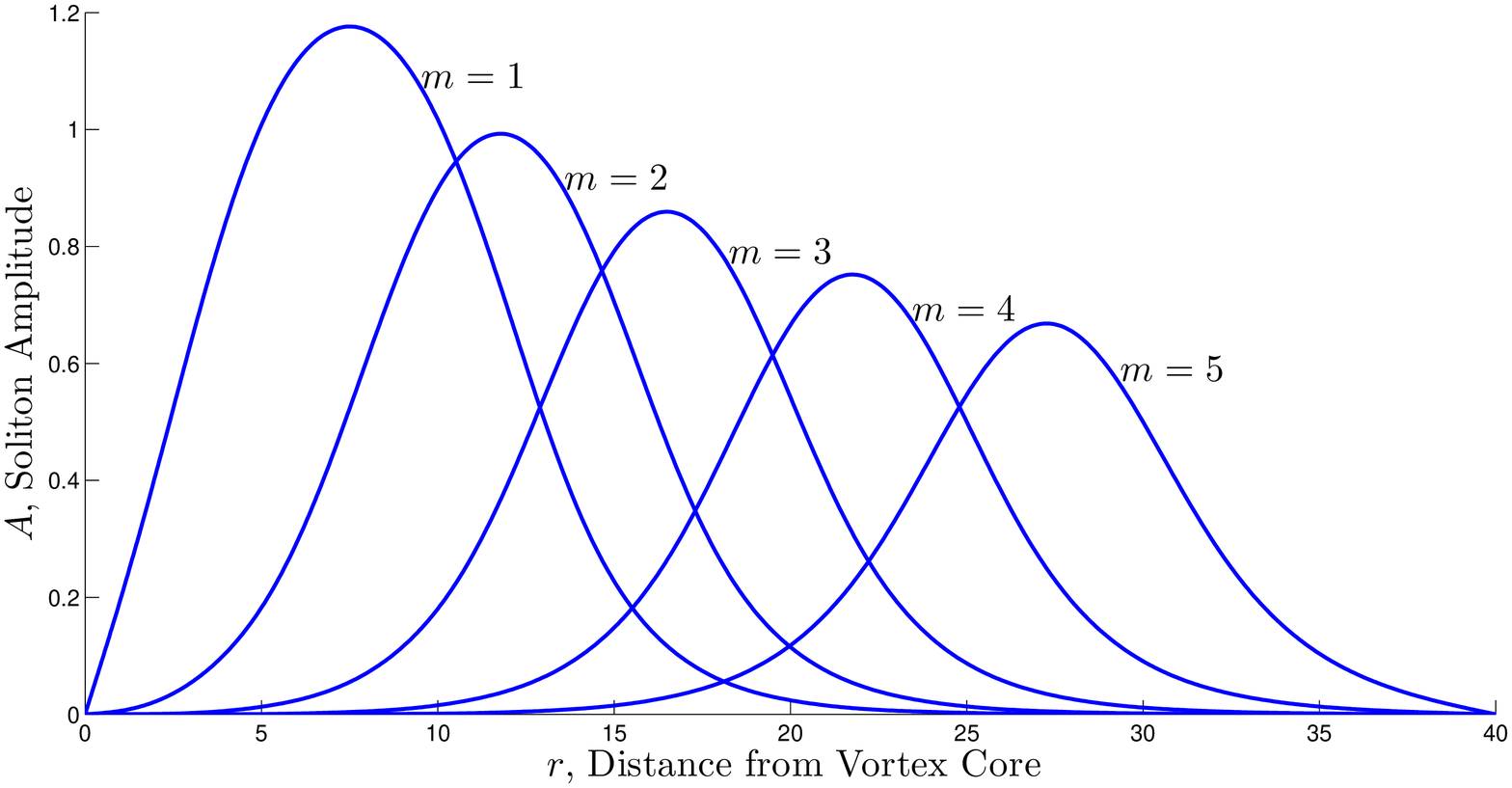}
\caption{Profile of single and multiple charged soliton solutions.}
\label{fig:fig2}
\end{figure*}

\begin{table*}[h]
\centering
\label{table2}
\begin{tabular}{|l|l|l|l|l|l|}
\hline
$m$ & 1 & 2 & 3 & 4 & 5 \\ \hline
$\lambda$ & 0.1816 & 0.1599 & 0.1388 & 0.119 & 0.1009 \\ \hline
$A_{max}$ & 1.1761 & 0.9929 & 0.8594 & 0.7520 & 0.6683\\ \hline
%error & 0.0355 & 0.0273 & 0.0245 & 0.0248 & 0.0192 \\ \hline
\end{tabular}
\caption{Nonlinear frequency shift and maximum amplitude of multiple charged soliton solutions.}
\end{table*} 

\section{Conclusions}
\setcounter{equation}{0}
A series of existence and computational results were obtained describing ``ring-profiled'' localized optical vortex solitons considered in the context of an electromagnetic pulse propagating in an arbitrary pair plasma. The electromagnetic pulse is modelled by a nonlinear Schr\"odinger equation \eqref{0.1} and the structure of the medium described by a new type of focusing-defocusing nonlinearity \eqref{0.2}. Such waves, which do not change its intensity profile during propagation, are governed by the $m$-vortex equation \eqref{vortexEq}. We have the following results,

\begin{enumerate}
\item A necessary condition for the existence of nontrivial solutions of the $m$-vortex equation \eqref{vortexEq2} is that the nonlinear frequency shift (or wave propagation constant) satisfies
\begin{align}
\lambda \leq f_{\max}-\dfrac{r_0^2+m^2}{R^2},
\end{align}
where $f_{\max}=\max\{|f(I)|:I\in[0,\infty)\}<\infty$, $r_0(\approx 2.404825)$, $m$ is the topological charge, and $R$ the distance from the ``vortex core'' (see Theorem 2.1). Localized optical vortex solitons are exponentially decaying solutions of \eqref{vortexEq2} and require the nonlinear frequency shift to be positive (see Lemma 2.2). Therefore, in the case of the focusing-defocusing nonlinearity \eqref{0.2}, a necessary condition for nontrivial localized optical vortex solitons is
\begin{align}
0<\lambda \leq \dfrac{1}{4}-\dfrac{r_0^2+m^2}{R^2}.
\end{align}

\item Using a direct minimization, we proved the existence of positive solution pairs $(A,\lambda)$ of \eqref{vortexEq}, where the nonlinear frequency shift may be prescribed in a continuous range of values. Specifically, for any $\lambda$ in $(0,\frac{3}{2}k_{max})$, $|m|\geq 1$, and
\begin{align}
\dfrac{6(1+m^2(2\ln(2)-1))}{\lambda}\leq R^2,
\end{align}
there exists a positive solution of \eqref{vortexEq} (see Theorem 3.1). From our approach, the best range of values for $\lambda$ we may obtain is $0<\lambda<3k_{max}\approx 0.2026221$. In \cite{BMS}, it was numerically found that $\lambda$ may be prescribed in $0<\lambda<\lambda_{cr}\approx 0.2162$. Our range of values for $\lambda$ covers more than 93.7\% of the one predicted in \cite{BMS} (see remarks in section 3).

\item A constrained minimization, where the constraint represents the photon number $N_0$, was also used to prove the existence of positive solutions of \eqref{0.2} (see Theorem 4.1). In this approach, the nonlinear frequency shift is undetermined, however, it provides some useful estimates on the nonlinear frequency shift and the prescribed photon number. For example, we have the non-existence of nontrivial ``small-photon-number'' pulses satisfying $N_0\leq 1/2$ whenever $m^2+r^2\lambda>0$ for $r\in[0,R]$ (see Theorem 4.2 $(a)$). Additionally, if $(A,\lambda)$ is a nontrivial solution pair of \eqref{vortexEq}, then $\lambda<0$ whenever $|m|\geq \frac{N_0}{2\pi}$ (see Theorem 4.2 $(b)$). Moreover, for any nontrivial solution pair of \eqref{vortexEq}, a lower bound is obtained for the nonlinear frequency shift in terms of the relevant parameters describing the pulse (see Theorem 4.2 $(c)$).

\item Lastly, we used the constrained minimization and a finite element formalism to compute the soliton profile of single and multiple charged soliton solutions as a function of the photon number. Our numerical results indicate that both the soliton amplitude and nonlinear frequency shift respectively increase to a maximum value of $A_{cr}\approx 1.5$ and $\lambda_{cr}\approx 0.21$ for an increasing photon number (see Figure 1 and Table 1). On the other hand, for a fixed photon number and increasing topological charge, both the nonlinear frequency shift and soliton amplitude decrease (see Figure 3 and Table 2). Interestingly, we also observed the so called ``flat-top'' solitons when the photon number is large (see Figure 2). 
\end{enumerate}

\small{

}

\end{document}